\documentclass[11pt,oneside,english]{amsart}
\usepackage{ae,aecompl}
\usepackage[T1]{fontenc}
\usepackage[latin9]{inputenc}
\usepackage{amsthm}
\usepackage{amstext}
\usepackage{amssymb}
\usepackage{esint}

\makeatletter
\numberwithin{equation}{section}
\numberwithin{figure}{section}
\theoremstyle{plain}
\newtheorem{thm}{\protect\theoremname}
  \theoremstyle{definition}
  \newtheorem{defn}{\protect\definitionname}
  \theoremstyle{plain}
  \newtheorem{lem}{\protect\lemmaname}
  \theoremstyle{plain}
  \newtheorem*{thm*}{\protect\theoremname}
  \theoremstyle{plain}
  \newtheorem{prop}{\protect\propositionname} 
  \newtheorem{cor}{\protect\corollaryname}
  \theoremstyle{definition}
  \newtheorem*{example}{\protect\examplename}
  \theoremstyle{remark}
  \newtheorem{rem}{\protect\remarkname}

\makeatother

\usepackage{babel}
  \providecommand{\definitionname}{Definition}
  \providecommand{\examplename}{Example}
  \providecommand{\lemmaname}{Lemma}
  \providecommand{\propositionname}{Proposition}
  \providecommand{\corollaryname}{Corollary}
  \providecommand{\remarkname}{Remark}
  \providecommand{\theoremname}{Theorem}

\begin{document}

\title{Prime Poisson suspensions}

\author{François Parreau and Emmanuel Roy}

\curraddr{Laboratoire Analyse Géométrie et Applications, UMR 7539, Université
Paris 13, 99 avenue J.B. Clément, F-93430 Villetaneuse, France}

\email{parreau@math.univ-paris13.fr, roy@math.univ-paris13.fr}
\begin{abstract}
We establish a necessary and sufficient condition for a Poisson suspension
to be prime. The proof is based on the Fock space structure of the
$L^{2}$-space of the Poisson suspension. We give examples of explicit
infinite measure preserving systems, in particular of non-singular compact group
rotations that give rise to prime Poisson suspensions. We
also compare some properties of so far known prime transformations with those
of our examples, showing that these examples are new.
\end{abstract}


\maketitle

\section{Introduction}

The aim of this paper is to build new examples of prime dynamical
systems, that is, systems which have a trivial factor
structure: their only factors are  (up to isomorphism) the original
system and the one point system.

The systems we construct are particular cases of Poisson suspensions which are
probability preserving system canonically build from infinite measure
preserving systems.

The paper is organized as follows: we first recall basic notions
on factors and give an overview of prime systems that exist in the
literature. Then in Section 2 we introduce Poisson measures and all the specific
tools that we shall need to derive structural results on the $\sigma$-algebra
of a Poisson measure. Section 3 is devoted to Poisson suspensions,
which are Poisson measures endowed with particular measure preserving
transformations, and in Section 4, we give the main result that will
be used to exhibit prime Poisson Suspensions. We also give a collection of ergodic
consequences of primeness for Poisson suspensions, in particular disjointness properties
from other families of dynamical systems.

The last Section gives concrete examples
of infinite measure systems that satisfy the hypothesis required to
obtain prime suspensions and prove that these systems actually exist.

\subsection{On factors of a dynamical system with an invariant probability measure}

We first recall what a factor of a dynamical system is. Let a system $\left(X,\mathcal{A},\mu,T\right)$ be given, where $\left(X,\mathcal{A},\mu\right)$
is a standard Borel probability space and $T$ is an invertible measure-preserving transformation on $X$. There are
two equivalent points of view:
\begin{enumerate}
\item We say that a sub-$\sigma$-algebra $\mathcal{C}\subset\mathcal{A}$ is a \emph{factor} if $T^{-1}\mathcal{C}=\mathcal{C}$.
\item Another system $\left(Y,\mathcal{B},\nu,S\right)$ is said
to be a \emph{factor} of $\left(X,\mathcal{A},\mu,T\right)$ if there is a measurable \emph{factor map} $\varphi: X\to Y$ such
that $\nu=\mu\circ\varphi^{-1}$ and the following diagram is commutative:
\[
\begin{array}{ccc}
\left(X,\mathcal{A},\mu\right) & \begin{array}[b]{c}
T\\
\rightarrow
\end{array} & \left(X,\mathcal{A},\mu\right)\\
\varphi\downarrow &  & \downarrow\varphi\\
\left(Y,\mathcal{B},\nu\right) & \begin{array}[b]{c}
S\\
\rightarrow
\end{array} & \left(Y,\mathcal{B},\nu\right)
\end{array}
\]

\end{enumerate}

These two definitions can be unified, indeed, with the latter one we obtain
a factor in the first sense by observing that $\varphi^{-1}\mathcal{B}$
satisfies $\varphi^{-1}\mathcal{B}\subset\mathcal{A}$ and $T^{-1}\left(\varphi^{-1}\mathcal{B}\right)=\varphi^{-1}\mathcal{B}$. For the other direction, if $\mathcal{B}\subset\mathcal{A}$
is a factor, we can form the quotient system $\left(\left(X_{\diagup\mathcal{B}},\mathcal{A}_{\diagup\mathcal{B}},\mu_{\diagup\mathcal{B}},T_{\diagup\mathcal{B}}\right)\right)$
and observe that the natural projection $\pi_{\mathcal{B}}$ is a
factor map.

We give the definition of a prime system with the first, internal,
point of view:
\begin{defn}
A system $\left(X,\mathcal{A},\mu,T\right)$ is said to be \emph{prime}
if $\mathcal{A}$ and $\left\{ X,\emptyset\right\} $ are the only
factors of the system.
\end{defn}
In other words, prime systems are those systems with the simplest factor
structure.

\subsection{An overview of previously known prime systems.}

The first examples are Ornstein's mixing rank one constructions (\cite{Orns72root}),
proved to be prime by Polit in \cite{Polit74weakiso}. Indeed those
systems are part of the larger class of \emph{simple systems} (\cite{Veech82prime, JuncRud87graphs})
which possess their own theory: they are those systems $\left(X,\mathcal{A},\mu,T\right)$
whose ergodic selfjoinings are either the product joining or graph
joinings $\Delta_{S}$ with $S\in C\left(T\right)$, the centralizer
of $T$. In particular factors $\mathcal{K}$ of simple systems correspond to compact
groups of $K\subset C\left(T\right)$ as follows:
\[
\mathcal{K}:=\left\{ A\in\mathcal{A},\: SA=A,\:\mbox{for all } A\in K\right\}. 
\]

Therefore, if $K$ is a maximal compact subgroup of the centralizer
of a simple system $T$, then it induces a prime system $T_{\diagup\mathcal{K}}$.
The most drastic situation occurs when the centralizer of a simple
system is reduced to the powers of the transformation, and thus the system itself is prime.
It is then said to have minimal self-joinings (MSJ(2)). Mixing rank one transformations
are such (\cite{King88Struc}), and so is Chacon transformation (\cite{JunRaSwa80Chacon})
and many others (see for example \cite{deljunco83counter}).

There also exist examples of rigid, and therefore not MSJ(2), simple
prime transformations (see \cite{DelRu87Rankonerig}).  In \cite{GlaWei94Prime} (see also \cite{DelDan08Cut}), examples
are given of simple systems with a centralizer possessing a non normal maximal
compact subgroup $K$, giving way to non simple prime systems since, if $T$ is simple, the factor system $T_{\diagup\mathcal{K}}$ is simple if and only if $K$ is
normal in $C\left(T\right)$.

We recall that there exist horocycle flows whose non-zero
times maps are prime. However, it is proved in \cite{Thou95Joinings}
that they can always be seen as factors of simple systems, and they are
thus part of the above theory.

Many examples of prime maps are also rank one. Indeed, mildly
mixing rank one maps are prime, as King showed in \cite{King86Comweak}
that a strict factor of a rank one map is rigid. It is yet unknown
wether prime rank one maps are always factors of simple systems.
To be complete, let us mention another situation  which do not fit into already described
ones, where prime systems occur:
if $T$ is
MSJ(2), then the symmetric factor of $T\times T$ is prime and so
is the map $\left(x,y\right)\mapsto\left(y,Tx\right)$ (see \cite{Rud90Fund} page 128).

We believe that we have listed, if not every example, at least all so far known families of prime probability measure preserving transformations.

\section{Technology}

\subsection{Poisson measure}

Let $\left(X,\mathcal{A},\mu\right)$ be a $\sigma$-finite, infinite measure
space, $\mu$ being continuous, and $\left(X^{*},\mathcal{A}^{*},\mu^{*}\right)$
be the corresponding \emph{Poisson measure space}. We recall the definition.

$X^{*}$ is the space of counting measures on $\left(X,\mathcal{A}\right)$,
i.e. measures of the form $\nu={\displaystyle \sum_{i\in I}}\delta_{x_{i}}$,
where the $x_{i}$ are in $X$ and $I$ is countable.
 
We define the maps $N\left(A\right)$, $A\in\mathcal{A}$, on
$X^{*}$ by $N\left(A\right)\left(\nu\right)=\nu\left(A\right)$.
The map $N\left(A\right)$ ``counts'' the number of points that
fall into the set $A$.
We set $\mathcal{A}^{*}:=\sigma\left\{ N\left(A\right),\: A\in\mathcal{A}\right\} $,
the smallest $\sigma$-algebra on $X^{*}$ that makes the operation
of measuring sets in $\mathcal{A}$ measurable.

The measure $\mu^{*}$ is now defined as the only probability measure on $\left(X^{*},\mathcal{A}^{*}\right)$
such that, for any $k\in\mathbb{N}$ and arbitrary disjoint sets $A_{1},\dots,A_{k}$
in $\mathcal{A}$ so that $0<\mu\left(A_{i}\right)<+\infty$, the
random variables $N\left(A_{1}\right),\dots,N\left(A_{k}\right)$
are independent and distributed according to Poisson laws with parameters
$\mu\left(A_{1}\right)$,\ldots, $\mu\left(A_{k}\right)$ respectively.

The underlying measure $\mu$ is often called the \emph{intensity}
of the Poisson measure.

It is frequent to denote the identity on $X^{*}$ by $N$, that is
$N\left(\nu\right)=\nu$ where $\nu$ is a counting measure on $X$.
Under the distribution $\mu^{*}$, $N$ is therefore a random measure.

\subsection{Fock space}

In the sequel, for $n\ge1$, let $L_{\text{sym}}^{2}\left(\mu^{\otimes n}\right)$
denote the subspace of $L^{2}\left(\mu^{\otimes n}\right)$ of functions
invariant by permutations of coordinates. It is convenient to endow
it with the normalized scalar product $\left\langle \cdot,\cdot\right\rangle _{n}:=\frac{1}{n!}\left\langle \cdot,\cdot\right\rangle _{L^{2}\left(\mu^{\otimes n}\right)}$.

We recall that $L^{2}\left(\mu^{*}\right)$ has a Fock-space structure
based on $L^{2}\left(\mu\right)$. Namely:

\[
L^{2}\left(\mu^{*}\right)\simeq\mathbb{C}\oplus L^{2}\left(\mu\right)\oplus L_{\text{sym}}^{2}\left(\mu^{\otimes2}\right)\oplus\cdots\oplus L_{\text{sym}}^{2}\left(\mu^{\otimes n}\right)\oplus\cdots
\]

That is, $L^{2}\left(\mu^{*}\right)$ can be seen as an orthogonal
sum of subspaces $H^{n}$, $n\in\mathbb{N}$, where $H^{n}$ (called \emph{chaos of
order $n$}) is naturally identified to $L_{\text{sym}}^{2}\left(\mu^{\otimes n}\right)$
(and $H^{0}$, the subspace of constant functions, is identified to $\mathbb{C}$). The identification
is done
 through the (normalized) multiple stochastic integrals $\frac{1}{n!}I^{\left(n\right)}$
from $L_{\text{sym}}^{2}\left(\mu^{\otimes n}\right)$ to $H^{n}$
whose constructions are detailed in \cite{Lieb94Isom} for example.

In this paper, we only need to know $I^{(n)}$ explicitely in the case $n=1$:
for $f\in L^{1}\left(\mu\right)\cap L^{2}\left(\mu\right)$, define
\[
I^{\left(1\right)}\left(f\right):=\int_{X}f\left(x\right)\left(N\left(dx\right)-\mu\left(dx\right)\right)
\]

\subsection{Difference operators}

In this section, we collect some information about \emph{difference
operators} we shall need in the next section. It is taken from
the very noticeable paper of Last and Penrose \cite{LastPen11Fock}.
Lemmas \ref{lem:aeequal} and \ref{lem:annihilChaos} is simply a convenient reformulation of 
the various properties of difference operators we need to obtain Theorem
\ref{thm:PrimePoisson}. They also appear somehow implicitely in \cite{LastPen11Fock}. 

Let $F$ be a measurable function on $\left(X^{*},\mathcal{A}^{*}\right)$.
Define the difference operator $D_{y}^{1}F$ by:
\[
D_{y}^{1}F\left(\nu\right):=F\left(\nu+\delta_{y}\right)-F\left(\nu\right).
\]
It consists into adding a particle at the position $y\in X$ and evaluating
$F$ with and without this particle and taking the difference.

By induction, we define $D_{y_{1}\dots,y_{n}}^{n}F$
\[
D_{y_{1}\dots,y_{n}}^{n}F:=D_{y_{2}\dots,y_{n}}^{n-1}\left(D_{y_{1}}^{1}F\right).
\]
It can be observed that this operator is symmetric in $y$'s for any
$n\in\mathbb{N}$.

Closely related to these operators is the following formula (known
as \emph{Mecke's formula} \cite{Meck67Form}) :

\begin{equation}
\int_{X^{*}}\int_{X}h\left(\nu,x\right)\nu\left(dx\right)\mu^{*}\left(d\nu\right)=\int_{X^{*}}\int_{X}h\left(\nu+\delta_{x},x\right)\mu\left(dx\right)\mu^{*}\left(d\nu\right)\label{eq:Palm}
\end{equation}
valid for all positive measurable functions $h$ defined on $X^{*}\times X$.

As a first consequence of this formula, we recall the following
result which is an immediate extension of Lemma 2.4 in \cite{LastPen11Fock}:
\begin{lem}\label{lem:aeequal}
If $F$ and $G$ are two measurable functions on $\left(X^{*},\mu^{*}\right)$
that coïncide $\mu^{*}$-almost surely, then for any $n\in\mathbb{N}^{*}$,
\[
D_{y_{1}\dots,y_{n}}^{n}F\left(\nu\right)=D_{y_{1}\dots,y_{n}}^{n}G\left(\nu\right),
\]

for $\mu^{*}\otimes\mu^{\otimes n}$-almost every $\left(\nu,y_{1}\dots,y_{n}\right)\in X^{*}\times X^{n}$.
\end{lem}
So, these operators are well defined for $F\in L^{2}\left(\mu^{*}\right)$. It turns out that
they establish a remarkable link with the Fock space structure (see \cite{LastPen11Fock}):
\begin{itemize}
\item For $\mu^{\otimes n}$-almost every $\left(y_{1}\dots,y_{n}\right)\in X^{n}$, $\nu\mapsto\left(y_{1}\dots,y_{n}\right)$,
$D_{y_{1}\dots,y_{n}}^{n}F\left(\nu\right)$ is $\mu^{*}$-integrable.
\item $\left(y_{1}\dots,y_{n}\right)\mapsto\mathbb{E}\left[D_{y_{1}\dots,y_{n}}^{n}F\right]$
is in $L_{\text{sym}}^{2}\left(\mu^{\otimes n}\right)$.
\item If we set $P_{n}F\left(y_{1}\dots,y_{n}\right):=\mathbb{E}\left[D_{y_{1}\dots,y_{n}}^{n}F\right]$,
and $P_{0}F:=\mathbb{E}\left[F\right]$, $F$ decomposes in the Fock
space as:
\[
F\simeq P_{0}F+\dots+P_{n}F+\dots
\]
\end{itemize}
In particular, for $f\in L_{\text{sym}}^{2}\left(\mu^{\otimes n}\right)$,  $P_{n}\left[\frac{1}{n!}I^{\left(n\right)}\left(f \right)\right]=f$,
$\mu^{\otimes n}$-a.e.,
and in this case we can even remove the expectation:
\[
D_{y_{1}\dots,y_{n}}^{n}\left[\frac{1}{n!}I^{\left(n\right)}\left(f\right)\right]\left(\nu\right)=f\left(y_{1}\dots,y_{n}\right)
\]
for $\mu^{*}\otimes\mu^{\otimes n}$-almost all $\nu,y_{1}\dots,y_{n}\in X^{*}\times X^{n}$.
\begin{lem}
\label{lem:annihilChaos}If $F\in H^{n}$, then, for $\mu$-almost
all $y\in X$, $D_{y}^{1}F\in H^{n-1}$.\end{lem}
\begin{proof}
We first claim that $D_{y}^{1}F$ is in $L^{2}\left(\mu^{*}\right)$ for $\mu$-almost
all $y\in X$ when $F\in H^{n}$. Then Equation (3.9) in \cite{LastPen11Fock} reduces to $F=I^{(n)}\left(f_{n}\right)$, where $f_n=\frac{1}{n!}P_{n}F\in  L_{\text{sym}}^{2}\left(\mu^{\otimes n}\right)$, and condition (3.11) is satified. Therefore, by Theorem 3.3 in \cite{LastPen11Fock}, $D_{y}F(\nu)=D'_{y}F(\nu)$ $\mu^{*}\otimes\mu$-a.e.$(\nu,y)$, where $D'_{y}$ is given by formula (3.10) of this  paper, that is in our case $D'_{y}F=nI^{(n-1)}f_n(\cdot,\ldots,\cdot,y)$.
The claim follows.

Now, let $k\neq n-$1, we get 
\begin{eqnarray*}
P_{k}\left(D_{y}^{1}F\right)\left(y_{1}\dots,y_{k}\right)
=\mathbb{E}\left[D_{y_{1}\dots,y_{k}}^{k}\left(D_{y}^{1}F\right)\right]
&=&\mathbb{E}\left[D_{y_{1}\dots,y_{k},y}^{k+1}F\right]\\
&=&P_{k+1}F\left(y_{1}\dots,y_{k},y\right).
\end{eqnarray*}
But as $F\in H^{n}$, $P_{k+1}F$ is zero $\mu^{\otimes k+1}$-a.e.,
and we deduce that, for $\mu$-almost all $y\in X$, $P_{k}\left(D_{y}^{1}F\right)$
is zero $\mu^{\otimes k}$-a.e..
This proves the Lemma.
\end{proof}

\subsection{A result on $\sigma$-algebras}

Let us present the key tool that will be applied
to prove the main result of this paper (Theorem \ref{thm:Main}).
\begin{thm}
\label{thm:PrimePoisson}Let $\left(X^{*},\mathcal{A}^{*},\mu^{*}\right)$
a Poisson measure and let $\Phi$ be a conditional expectation on
a $\sigma$-algebra $\mathcal{C}\subset\mathcal{A}^{*}$ that preserves
the $n$-th chaos $H^{n}$ for every $n\ge1$. If $\Phi$ is zero
on $H^{1}$, then $\Phi$ is zero on $H^{n}$ for every $n\ge1$.
In other words, $\Phi$ is the conditional expectation on the trivial
$\sigma$-algebra $\left\{ X^{*},\emptyset\right\} $.\end{thm}
\begin{proof}
Let $F$ be in $H^{n}$, $n\ge2$, we want to show that $\Phi F=0$.
Without loss of generality we can assume that $F$ is real.
As $\Phi F\in H^{n}$, we have $P_{k}\left(\Phi F\right)=0$ for all $k\neq n$, and therefore it is enough to show that $P_{n}\left(\Phi F\right)=0$.

We shall prove that $D_{a}^{1}\Phi F(\nu)=0$ for $\mu^{*}\otimes\mu$-a.e.$\left(\nu,a\right)$.
It will follow from Lemma \ref{lem:aeequal} that $D_{y_{1}\dots,y_{n}}^{n}\Phi F(\nu)=0$
for $\mu^{*}\otimes\mu^{\otimes n}$-a.e.$\left(\nu,y_{1},\dots,y_{n}\right)$, so  
$P_{n}\left(\Phi F\right)=0$, and the proof will be complete.

Let $a\in X$,
\begin{eqnarray*}
\mathbb{E}\left[\left(D_{a}^{1}\Phi F\right)^{2}\right]
&=&\mathbb{E}\left[\left(\Phi F\left(\cdot+\delta_{a}\right)-\Phi F\right)^{2}\right]\\
&=&\mathbb{E}\left[\Phi F\left(\cdot+\delta_{a}\right)^{2}\right]+\mathbb{E}\left[\left(\Phi F\right)^{2}\right]-2\mathbb{E}\left[\Phi F\left(\cdot+\delta_{a}\right)\Phi F\right].
\end{eqnarray*}
We have
\begin{eqnarray*}
\mathbb{E}\left[\Phi F\left(\cdot+\delta_{a}\right)\Phi F\right]
& = & \mathbb{E}\left[\left(\Phi F\left(\cdot+\delta_{a}\right)-\Phi F\right)\Phi F\right]
+\mathbb{E}\left[\left(\Phi F\right)^{2}\right]\\
& = & \mathbb{E}\left[D_{a}^{1}\Phi F
\cdot \Phi F\right]+\mathbb{E}\left[\left(\Phi F\right)^{2}\right].
\end{eqnarray*}
But, as $\Phi F$ is in $H^{n}$, $D_{a}^{1}\Phi F$ is in $H^{n-1}$
for $\mu$-almost all $a\in X$, thanks to Lemma \ref{lem:annihilChaos}.
These two vectors are therefore orthogonal which means that $\mathbb{E}\left[\Phi F\left(\cdot+\delta_{a}\right)\Phi F\right]=\mathbb{E}\left[\left(\Phi F\right)^{2}\right]$. So
\[
\mathbb{E}\left[\left(D_{a}^{1}\Phi F\right)^{2}\right]
=\mathbb{E}\left[\Phi F\left(\cdot+\delta_{a}\right)^{2}\right]-\mathbb{E}\left[\left(\Phi F\right)^{2}\right].
\]

Now, we apply Mecke's formula (\ref{eq:Palm}) with $h\left(\nu,x\right)=\left(\Phi F\right)^{2}\left(\nu\right)f\left(x\right)$,
where $f$ is a nonnegative function in $L^{1}\left(\mu\right)\cap L^{2}\left(\mu\right)$. We get
\[
\int_{X^{*}}\int_{X}\Phi F\left(\nu\right)^{2}f\left(x\right)\nu\left(dx\right)\mu^{*}\left(d\nu\right)
=\int_{X^{*}}\int_{X}\Phi F\left(\nu+\delta_{x}\right)^{2}f\left(x\right)\mu\left(dx\right)\mu^{*}\left(d\nu\right)
\]
which can be rewritten
\[
\mathbb{E}\left[\left(\Phi F\right)^{2}\int_{X}f\left(x\right)N\left(dx\right)\right]
=\int_{X}\mathbb{E}\left[\Phi F\left(\cdot+\delta_{x}\right)^{2}\right]f\left(x\right)\mu(dx).
\]
Since  $I^{\left(1\right)}\left(f\right)=\int_{X}f\left(x\right)N\left(dx\right)-\int_{X}f\left(x\right)dx$, we also have
\[
\mathbb{E}\left[\left(\Phi F\right)^{2}\int_{X}f\left(x\right)N\left(dx\right)\right]
=\mathbb{E}\left[\left(\Phi F\right)^{2}I^{\left(1\right)}\left(f\right)\right]+\int_{X}\mathbb{E}\left[\left(\Phi F\right)^{2}\right]f\left(x\right)\mu(dx).
\]
As $\Phi$ is the conditional expectation on $\mathcal{C}$, $\Phi F$
is $\mathcal{C}$-measurable and so is $\left(\Phi F\right)^{2}$. But, by assumption, $\Phi$ vanishes on $H^{1}$, which implies that
the conditional expectation of $I^{\left(1\right)}\left(f\right)$ on $\mathcal{C}$ is zero, and therefore
$\mathbb{E}\left[\left(\Phi F\right)^{2}I^{\left(1\right)}\left(f\right)\right]=0$.
Hence
\[
\mathbb{E}\left[\left(\Phi F\right)^{2}\int_{X}f\left(x\right)N\left(dx\right)\right]=\int_{X}\mathbb{E}\left[\left(\Phi F\right)^{2}\right]f\left(x\right)\mu(dx)
\]
and so we get
\[
\int_{X}\mathbb{E}\left[\left(\Phi F\right)^{2}\right]f\left(x\right)\mu(dx)=\int_{X}\mathbb{E}\left[\left(\Phi F\right)^{2}\left(\cdot+\delta_{x}\right)\right]f\left(x\right)\mu(dx).
\]
As this equality holds for any nonnegative $f\in L^{1}\left(\mu\right)\cap L^{2}\left(\mu\right)$, we obtain
\[
\mathbb{E}\left[\left(\Phi F\right)^{2}\right]=\mathbb{E}\left[\left(\Phi F\right)^{2}\left(\cdot+\delta_{x}\right)\right]
\]
for $\mu$-almost all $x\in X$.

Summing up, for $\mu$-almost all $a\in X$,
\[
\mathbb{E}\left[\left(D_{a}^{1}\Phi F\right)^{2}\right]=0
\]
and thus $D_{a}^{1}\Phi F\left(\nu\right)=0$ $\mu^{*}\otimes\mu$-a.e..
\end{proof}

\section{Poisson suspensions}

If $T$ is a measure preserving automorphism of $\left(X,\mathcal{A},\mu\right)$,
then $T_{*}:\nu\mapsto\nu\circ T^{-1}$ is
a measure preserving automorphism of $\left(X^{*},\mathcal{A}^{*},\mu^{*}\right)$.
$\left(X^{*},\mathcal{A}^{*},\mu^{*},T_{*}\right)$ is the \emph{Poisson
suspension} over the \emph{base} $\left(X,\mathcal{A},\mu,T\right)$.

Let us recall the most basic ergodic result about Poisson suspensions
(see \cite{Roy07Infinite} for a proof):
\begin{thm*}
\textup{$\left(X^{*},\mathcal{A}^{*},\mu^{*},T_{*}\right)$ is ergodic
(and then weakly mixing) if and only if $\left(X,\mathcal{A},\mu,T\right)$
has no $T$-invariant set $A$ with $0<\mu\left(A\right)<+\infty$.}
\end{thm*}
In particular, if $T$ is ergodic and $\mu$ is infinite, then $T_{*}$
is ergodic.

The ergodic theory of Poisson suspension deals with the interplay
between a dynamical system with a $\sigma$-finite measure $\left(X,\mathcal{A},\mu,T\right)$
and the canonically built probability measure preserving system $\left(X^{*},\mathcal{A}^{*},\mu^{*},T_{*}\right)$.

Two directions are to be considered: Poisson suspensions can be seen
as a probabilistic tool to study infinite ergodic theory and the other direction
is to look them as a family of probabilistic systems
indexed by infinite measure preserving ones and see what kind of properties
we get. This paper belongs to the latter category.

\subsection{Poissonian factors}

There are two main ways to obtain natural factors of a Poisson suspension
$\left(X^{*},\mathcal{A}^{*},\mu^{*},T_{*}\right)$.

First assume you can find a $T$-invariant measurable set $A\subset X$,
of positive measure. Then the Poisson measure restricted to $A$
is such a factor. Indeed, the map 
\begin{eqnarray*}
X^{*} & \to & A^{*}\\
\nu & \mapsto & \nu_{\mid A}
\end{eqnarray*}
realizes a factor map between $\left(X^{*},\mathcal{A}^{*},\mu^{*},T_{*}\right)$
and $\left(A^{*},\left(\mathcal{A}_{\mid A}\right)^{*},\left(\mu_{\mid A}\right)^{*},\left(T_{\mid A}\right)^{*}\right)$.

In terms of $\sigma$-algebra, the above factor corresponds to 
\[
\sigma\left\{ N\left(C\right),\: C\in\mathcal{A},\: C\subset A\right\} \subset\mathcal{A}^{*}.
\]

The second way consists in considering $\sigma$-finite factors of
the base (we mean $T$-invariant $\sigma$-algebras $\mathcal{B}\subset\mathcal{A}$ such that $\mu_{|\mathcal{B}}$ remains $\sigma$-finite).
Namely, if $\mathcal{B}$ is such a $\sigma$-finite
factor, we have the factor map
\[
\left(X,\mathcal{A},\mu,T\right)\begin{array}[b]{c}
\psi\\
\to
\end{array}\left(X_{\diagup\mathcal{B}},\mathcal{A}_{\diagup\mathcal{B}},\mu_{\diagup\mathcal{B}},T_{\diagup\mathcal{B}}\right)
\]
and if we define $\psi_{*}$ by $\nu\mapsto\nu\circ\psi^{-1}$
we obtain the factor relationship at the level of the Poisson suspensions:
\[
\left(X^{*},\mathcal{A}^{*},\mu^{*},T_{*}\right)\begin{array}[b]{c}
\psi_{*}\\
\to
\end{array}\left(\left(X_{\diagup\mathcal{B}}\right)^{*},\left(\mathcal{A}_{\diagup\mathcal{B}}\right)^{*},\left(\mu_{\diagup\mathcal{B}}\right)^{*},\left(T_{\diagup\mathcal{B}}\right)^{*}\right)
\]
In terms of $\sigma$-algebra, it corresponds to $\mathcal{B}^{*}:=\sigma\left\{ N\left(C\right),\: C\in\mathcal{B}\right\} \subset\mathcal{A}^{*}$.

A Poissonian factor is a combination of both situations which is obtained
by first considering a $T$-invariant subset $A\subset X$ and then
considering a $\sigma$-finite factor $\mathcal{B}$ of the restricted
system $\left(A,\mathcal{A}_{\mid A},\mu_{\mid A},T_{\mid A}\right)$.
We also consider the trivial factor $\left\{ \emptyset,X^{*}\right\} $
as a Poissonian factor.

In particular, if $\left(X,\mathcal{A},\mu,T\right)$ is ergodic, then
the only Poissonian factors of $\left(X^{*},\mathcal{A}^{*},\mu^{*},T_{*}\right)$
are:
\begin{itemize}
\item the trivial factor $\left\{ \emptyset,X^{*}\right\} $;
\item $\mathcal{B}^{*}$, for a $\sigma$-finite factor $\mathcal{B}\subset\mathcal{A}$.
\end{itemize}

We shall need a result from \cite{Roy07Infinite}. We recall that a {\em sub-Markov} operator on $L^{2}\left(\mu\right)$ is a positive operator
$\Phi$ such that $\Phi f\le 1$ and $\Phi^{*}f\le 1$, for $0\le f \le 1$.

\begin{prop}
\label{pro:conditionalexpectation}Let $\mathcal{C}\subset\mathcal{A}^{*}$
be a factor of $\left(X^{*},\mathcal{A}^{*},\mu^{*},T_{*}\right)$
and $\Phi$ the corresponding conditional expectation. Assume moreover
that $\Phi$ preserves the first chaos $H^{1}$ and does not vanish
on $H^{1}$. Then:
\begin{itemize}
\item $\Phi$ induces on $L^{2}\left(\mu\right)$ a sub-Markov operator
$\Psi$.
\item There exists a $T$-invariant set $A\subset X$ such that $\Psi$
restricted to $L^{2}\left(\mu_{\mid A}\right)$ is a conditional expectation
on a $\sigma$-finite factor $\mathcal{G}\subset\mathcal{A}_{\mid A}$
and vanishes on $L^{2}\left(\mu_{\mid A^{c}}\right)$.
\end{itemize}
\end{prop}

\subsection{\label{sub:Superposition-of-Poisson-Measures}Superposition of Poisson
suspensions}
We shall need later (for Propositions \ref{prop:sym} and \ref{prop:mutimeslambda}) an easy and classical fact about Poisson measures which says that if we consider two independent Poisson measures living on the same base, with intensities $\mu_{1}$ and $\mu_{2}$, the superposition of Poisson suspensions, that is the full collection of particles coming from
both systems gives birth to a Poisson measure with intensity $\mu_{1}+\mu_{2}$.
More precisely, we have the following map:
\begin{eqnarray*}
 & \left(X^{*}\times X^{*},\mathcal{A}^{*}\otimes\mathcal{A}^{*},\mu_{1}^{*}\otimes\mu_{2}^{*}\right)\\
 & \downarrow\Psi\\
 & \left(X^{*},\mathcal{A}^{*},\left(\mu_{1}+\mu_{2}\right)^{*}\right)
\end{eqnarray*}
defined by
\[
\Psi\left(\nu_{1},\nu_{2}\right)=\nu_{1}+\nu_{2}
\]

From an ergodic point of view, as $\Psi\circ\left(T_{*}\times T_{*}\right)=T_{*}\circ\Psi$,
the Poisson suspension $\left(X^{*},\mathcal{A}^{*},\left(\mu_{1}+\mu_{2}\right)^{*},T_{*}\right)$
can be seen as a factor of the direct product 
$\left(X^{*},\mathcal{A}^{*},\mu_{1}^{*},T_{*}\right)\times\left(X^{*},\mathcal{A}^{*},\mu_{2}^{*},T_{*}\right)$. If $\mu_{1}=\mu_{2}=\mu$, as $\Psi$ is symmetric, it is a factor
of the symmetric factor of this direct product (the $\sigma$-algebra
$\mathcal{A}^{*}\odot\mathcal{A}^{*}\subset\mathcal{A}^{*}\otimes\mathcal{A}^{*}$
generated by symmetric functions on $\left(X^{*}\times X^{*},\mathcal{A}^{*}\otimes\mathcal{A}^{*},\mu^{*}\otimes\mu^{*}\right)$).

\subsection{Unitary operators $U_{T}$ and $U_{T_{*}}$}

$T$ acts unitarily on $L^{2}\left(\mu\right)$ by $U_{T}:f\mapsto f\circ T$
and so does $T_{*}$ on $L^{2}\left(\mu^{*}\right)$ by $U_{T_{*}}:F\mapsto F\circ T_{*}$.
Each chaos is preserved by $U_{T_{*}}$ and, through the above identification,
it is easy to see that it corresponds to $U_{T}$ on $H^{1}\simeq L^{2}\left(\mu\right)$,
and more generally to $U_{T}^{\odot n}$ (the $n$-th symmetric tensor power of $U_{T}$) on $H^{n}\simeq L_{\text{sym}}^{2}\left(\mu^{\otimes n}\right)$.

If $\sigma$ is the maximal spectral type of $U_{T}$ on $L^{2}\left(\mu\right)$
then $\sigma^{*n}$ is the maximal spectral type of $U_{T}^{\odot n}$
on $L_{\text{sym}}^{2}\left(\mu^{\otimes n}\right)$.

We shall need another definition:
\begin{defn}
A Poisson suspension is said to have the {\em property CP} (for ``chaos-preserving'')
if any conditional expectation with respect to a factor preserves
each chaos $H^{n}$.
\end{defn}
Below is the main situation where we obtain this property.
\begin{example}
If the maximal spectral type $\sigma$ of $\left(X,\mathcal{A},\mu,T\right)$
satisfies $\sigma^{*n}\perp\sigma^{*m}$, for all distinct $n$, $m\in\mathbb{N}^{*}$,
then $\left(X^{*},\mathcal{A}^{*},\mu^{*},T_{*}\right)$
has property CP (see \cite{LemParThou00Gausselfjoin} for the proof
in the Gaussian case, the Poissonian one is completely analogous).
\end{example}

\section{\label{sec:The-Main-result}Prime Poisson Suspensions}
\subsection{Main result}
We have to extend the definition of prime systems to the case of an infinite measure.
The difference with the probability measure case resides in the fact
that the trivial factor $\left\{ X,\emptyset\right\} $ is no longer
$\sigma$-finite factor when the measure is infinite.
\begin{defn}
An ergodic system $\left(X,\mathcal{A},\mu,T\right)$ with an infinite
measure is said to be \emph{prime} if $\mathcal{A}$ is its only $\sigma$-finite
factor.
\end{defn}
We are now able to prove the main result of the paper.
We shall give examples of such systems in the last section.
\begin{thm}
\label{thm:Main}Let $\left(X,\mathcal{A},\mu,T\right)$ be an ergodic
infinite measure preserving system such that $\left(X^{*},\mathcal{A}^{*},\mu^{*},T_{*}\right)$
has property CP. If $\mathcal{C}\subset\mathcal{A}^{*}$ is a non-trivial
factor, then it contains a non-trivial Poissonian factor. In particular,
if we assume moreover that $\left(X,\mathcal{A},\mu,T\right)$ is
prime, then \textup{$\left(X^{*},\mathcal{A}^{*},\mu^{*},T_{*}\right)$
is prime.}\end{thm}
\begin{proof}
Let $\mathcal{C}$ be a non-trivial $T_{*}$-invariant $\sigma$-algebra
included in $\mathcal{A}^{*}$ and $\Phi$ the corresponding conditional
expectation, it preserves $H^{1}$ thanks to property CP.

Assume firstly that $\Psi$ vanishes on $H^{1}$, then we can apply
Theorem \ref{thm:PrimePoisson} to conclude that
$\mathcal{C}=\left\{ X^{*},\emptyset\right\} $ which is impossible
as we have assumed $\mathcal{C}$ to be a non-trivial factor.

Now if $\Psi$ doesn't vanish on $H^{1}$ we can apply Proposition \ref{pro:conditionalexpectation}, combined with the ergodicity of $T$ to deduce that $\Phi$ induces on $L^{2}\left(\mu\right)$ a sub-Markov operator $\Psi$ which is also conditional expectation on a $\sigma$-finite factor $\mathcal{T}$. The image of $\Psi$
contains all the indicator functions of finite measure sets contained
in $\mathcal{T}$. Coming back to $L^{2}\left(\mu^{*}\right)$, the image of $\Phi$ contains
all the vectors $I^{(1)}(1_A)=N\left(A\right)-\mu\left(A\right)$, for
$A\in\mathcal{T}$ of finite measure, which are therefore $\mathcal{C}$-measurable.
This proves that $\mathcal{C}$ contains the Poissonian factor $\mathcal{T}^{*}$.
\end{proof}
\begin{rem} The conditions of the Theorem are also necessary in order that the Poisson suspension be prime: then, we have no non-trivial conditional expectation with respect to a factor, so property CP holds obviously, and there is no proper non-trivial Poissonian factor so $\left(X,\mathcal{A},\mu,T\right)$ must be prime.
\end{rem}
\subsection{Some consequences}
We mentionned in the Introduction that when $T$ is MSJ(2), then $T\odot T$, the
symmetric factor of the direct product $T\times T$, is prime and
so is the map $\left(x,y\right)\mapsto\left(y,Tx\right)$ with respect
to the product measure. It is therefore natural to ask if this is
the case in our context.
\begin{prop}
\label{prop:sym}$T_{*}\odot T_{*}$ is never prime.\end{prop}
\begin{proof}
With the result on the superposition of two independent Poisson
measures recalled in Section \ref{sub:Superposition-of-Poisson-Measures},
we obtain the scheme
\begin{eqnarray*}
 & \left(X^{*}\times X^{*},\mathcal{A}^{*}\otimes\mathcal{A}^{*},\mu^{*}\otimes\mu^{*},T_{*}\times T_{*}\right)\\
 & \downarrow\\
 & \left(X^{*}\times X^{*},\mathcal{A}^{*}\odot\mathcal{A}^{*},\mu^{*}\otimes\mu^{*},T_{*}\odot T_{*}\right)\\
 & \downarrow\\
 & \left(X^{*},\mathcal{A}^{*},\left(2\mu\right)^{*},T_{*}\right)
\end{eqnarray*}
The direct product $\left(X^{*}\times X^{*},\mathcal{A}^{*}\otimes\mathcal{A}^{*},\mu^{*}\otimes\mu^{*},T_{*}\times T_{*}\right)$
can be thought as the Poisson suspension
\[
\left(\left(X\times X\right)^{*},\left(\mathcal{A}\otimes\mathcal{A}\right)^{*},\left(\mu\otimes\delta_{\infty}+\delta_{\infty}\otimes\mu\right)^{*},\left(T\times T\right)_{*}\right)
\]
where $\infty$ is an artificially added point in $X$, fixed by $T$.
In this way, $\left(X^{*},\mathcal{A}^{*},\left(2\mu\right)^{*},T_{*}\right)$
appears as a Poissonian factor (corresponding to the symmetric factor of $X\times X$), and we know that the corresponding
relatively independent joining is ergodic (see \cite{Roy07Infinite}). It follows that
\begin{eqnarray*}
 & \left(X^{*}\times X^{*},\mathcal{A}^{*}\otimes\mathcal{A}^{*},\mu^{*}\otimes\mu^{*},T_{*}\times T_{*}\right)\\
 & \downarrow\\
 & \left(X^{*},\mathcal{A}^{*},\left(2\mu\right)^{*},T_{*}\right)
\end{eqnarray*}
is a relatively weakly mixing extension. 

However we know that
\begin{eqnarray*}
 & \left(X^{*}\times X^{*},\mathcal{A}^{*}\otimes\mathcal{A}^{*},\mu^{*}\otimes\mu^{*},T_{*}\times T_{*}\right)\\
 & \downarrow\\
 & \left(X^{*}\times X^{*},\mathcal{A}^{*}\odot\mathcal{A}^{*},\mu^{*}\otimes\mu^{*},T_{*}\odot T_{*}\right)
\end{eqnarray*}
is a compact extension. This proves that $\left(X^{*},\mathcal{A}^{*},\left(2\mu\right)^{*},T_{*}\right)$
is a strict and non trivial factor of $\left(X^{*}\times X^{*},\mathcal{A}^{*}\odot\mathcal{A}^{*},\mu^{*}\otimes\mu^{*},T_{*}\odot T_{*}\right)$.\end{proof}

\begin{prop}
If $T_{*}$ is prime and $T_{*}\times T_{*}$ has property CP, then
$\left(\nu_{1},\nu_{2}\right)\mapsto\left(\nu_{2},T_{*}\nu_{1}\right)$
is prime.\end{prop}
\begin{proof}
We use the representation of $\left(X^{*}\times X^{*},\mathcal{A}^{*}\otimes\mathcal{A}^{*},\mu^{*}\otimes\mu^{*},T_{*}\times T_{*}\right)$
introduced in the above proof, that is
\[
\left(\left(X\times X\right)^{*},\left(\mathcal{A}\otimes\mathcal{A}\right)^{*},\left(\mu\otimes\delta_{\infty}+\delta_{\infty}\otimes\mu\right)^{*},\left(T\times T\right)_{*}\right)
\]
In this representation, $\left(\nu_{1},\nu_{2}\right)\mapsto\left(\nu_{2},T_{*}\nu_{1}\right)$
becomes $\left(\nu_{1}\otimes\delta_{\infty}+\delta_{\infty}\otimes\nu_{2}\right)\mapsto\left(\nu_{2}\otimes\delta_{\infty}+\delta_{\infty}\otimes T_{*}\nu_{1}\right)$
and is indeed a Poisson suspension over
\[
\left(X\times X,\mathcal{A}\otimes\mathcal{A},\mu\otimes\delta_{\infty}+\delta_{\infty}\otimes\mu,R\right)
\]
where $R$ maps $\left(x,y\right)$ to $\left(y,Tx\right)$.

Observe that $R^{2}=T\times T$ and that $R$ is ergodic. Indeed,
if $A$ is an $R$-invariant set, then it is also a $T\times T$-invariant
set. But, as $T$ is ergodic, it is easy to see that, modulo null sets with respect to
the measure $\mu\otimes\delta_{\infty}+\delta_{\infty}\otimes\mu$, the only $T\times T$-invariant
sets are $\emptyset$, $X\times X$, $X\times\left\{ \infty\right\} $
and $\left\{ \infty\right\} \times X$. Since the last two sets are obviously
not $R$-invariant, $R$ is ergodic.

In the same vein, a $\sigma$-finite factor of $R$ is also a $\sigma$-finite
factor of $T\times T$ and, with respect to
the measure $\mu\otimes\delta_{\infty}+\delta_{\infty}\otimes\mu$, the only $\sigma$-finite factor of $T\times T$ is the symmetric
factor which is not a factor of $R$, so $R$ is prime.

It remains to check that $\left(\left(X\times X\right)^{*},\left(\mathcal{A}\otimes\mathcal{A}\right)^{*},\left(\mu\otimes\delta_{\infty}+\delta_{\infty}\otimes\mu\right)^{*},R_{*}\right)$
has property CP. It follows from the assumption that $T_{*}\times T_{*}$ has property CP, hence so does  $\left(T\times T\right)_{*}$ under the measure $\left(\mu\otimes\delta_{\infty}+\delta_{\infty}\otimes\mu\right)^{*}$, and from the fact that $R_{*}^{2}=\left(T\times T\right)_{*}$.
\end{proof}
In particular, if the maximal spectral type $\sigma$ of $T$ satisfies $\sigma^{*n}\perp\sigma^{*m}$,
then $T_{*}\times T_{*}$ also has property CP, as $\sigma$ is still the maximal spectral type of $T\times T$ with respect to the measure $\mu\otimes\delta_{\infty}+\delta_{\infty}\otimes\mu$.

\subsection{Disjointness}

The following disjointness results come from \cite{LemPaRo11JoinPrim}
where the notion of Joining Primeness of order $n$ (JP($n$)) was
introduced. Simple maps and their factors are JP($1$) and direct
products of such maps are JP($2$).
\begin{thm}
\label{thm:ParreauLemRoy}\cite{LemPaRo11JoinPrim} A Poisson suspension
is disjoint from every JP($n$) map for any $n\ge1$.
\end{thm}
Therefore our prime Poisson suspensions are disjoint from prime maps
that are simple or factor of simple maps.
\begin{prop}\label{prop:disjointness}
If a transformation $S$ is distally simple, then $S\odot S$
and $K:=\left(x,y\right)\mapsto\left(y,Sx\right)$ are disjoint from
Poisson suspensions.
\end{prop}
\begin{proof}
The first point follows from the fact that $S\times S$ is JP(2) and
so is $S\odot S$.

For the second point, a non-trivial joining between $K$ and a Poisson
suspension $T_{*}$ would yield a non-trivial joining between $K^{2}=S\times S$
and $\left(T_{*}\right)^{2}=\left(T^{2}\right)_{*}$ which is impossible
by the above arguments.
\end{proof}

\section{\label{sec:Examples}Examples}

\subsection{Non-singular compact group rotations}

We introduce a family of examples that was studied, among other
sources, in \cite{Aar87Eigen} and \cite{HoMePa91nonsing}. We use the notation of the latter paper.

Consider the $2$-adding machine $\Omega:=\left\{ 0,1\right\} ^{\mathbb{N}}$ equipped with
the uniform Bernoulli probability measure $\nu:=\left(\frac{1}{2}\delta_{0}+\frac{1}{2}\delta_{1}\right)^{\otimes\mathbb{N}}$,
and the transformation $\omega\mapsto\omega+\overline{1}$, where $\overline{1}=\left(1,0,0,\dots\right)$ and addition is modulo $2$ with ``carrier to the right''.
Next, consider a measurable positive integer-valued function $h$ on $\omega$. 

We build the Kakutani tower over $\Omega$ with height function
$h$. Define $X\subset\Omega\times\mathbb{N}$ as the set of points $\left(\omega,n\right)$
such that $1\le n\le h\left(\omega\right)$ and let $T$ be the transformation on $X$ given by
\[
T\left(\omega,n\right)=\begin{cases}
\left(\omega,n+1\right) & \text{if}\:1\le n<h\left(\omega\right)\\
\left(\omega+\overline{1},1\right) & \text{if}\: n=h\left(\omega\right)
\end{cases}
\]
We endow $X$ with the $\sigma$-algebra $\mathcal{A}$ naturally inherited from the Borel $\sigma$-algebra of $\Omega$ and the $T$-invariant measure on $\left(X, \mathcal{A}\right)$ defined by
\[
\int_{X}f\left(\omega,n\right)\mu\left(d\left(\omega,n\right)\right)=\int_{\Omega}\left[\sum_{n=1}^{h\left(\omega\right)}f\left(\omega,n\right)\right]\nu\left(d\omega\right)
\]
for every measurable positive function $f$ on $X$.

Let us now give a more precise specification on $h$. Consider a sequence
of integers $\left\{ m_{i}\right\} _{i\ge0}$ where $m_{i}\ge3$ and
set $n_{i+1}=m_{i}n_{i}$ where $n_{0}=1$.
For $\omega\in\Omega$, let us denote $k\left(\omega\right)$ the
smallest integer $k$ such that $\omega_{k}=0$ and
\[
h\left(\omega\right)=n_{k\left(\omega\right)}-\sum_{j<k\left(\omega\right)}n_{j}.
\]
Then it is easy to see that the measure $\mu$ is infinite. It can be noted (see \cite{HoMePa91nonsing}) that this system encodes
an ergodic infinite measure preserving compact group rotation, namely
the adding machine on $\prod_{j\ge0}\left\{ 0,\dots,m_{j}-1\right\} $
endowed with a measure singular with respect to the Haar measure on this group.

For the sequel of this Section, $\left(X,\mathcal{A},\mu,T\right)$ denotes the above defined system.
\subsection{Properties}
 In \cite{Aar87Eigen}, the following is proved:
\begin{prop}
\label{pro:AarNad}Joinings between $\left(X,\mathcal{A},c_{1}\mu,T\right)$
and $\left(X,\mathcal{A},c_{2}\mu,T\right)$ exist only for $c_{1}=c_{2}$
and are graph joinings $\Delta_{T^{n}}$, $n\in\mathbb{Z}$. In particular,
the system is prime.
\end{prop}
In \cite{HoMePa91nonsing}, the spectrum of $T$ is determined:
\begin{prop}
The spectrum of $T$ is simple and its maximal specral type is the Riesz product 
\[
\sigma:={\displaystyle \prod_{j=0}^{+\infty}}\left(1+\cos2\pi n_{j}t\right)
\]

\end{prop}
Those Riesz products are the most classical ones and we have (\cite{BrMo74-Riesz}, \cite{Peyr74Riesz}):
\begin{prop}\label{pro:singpowers}
All the convolution powers $\sigma^{*n}$ are continuous and singular. Moreover,  $\sigma^{*n}\perp\sigma^{*m}$ for all $n\neq m$.
\end{prop}

\subsection{Poisson suspensions over $\left(X,\mathcal{A},\mu,T\right)$}

As a direct application, we obtain our first examples of prime Poisson
suspensions
\begin{prop}
The Poisson suspension $\left(X^{*},\mathcal{A}^{*},\mu^{*},T_{*}\right)$
is prime. Moreover, it is mildly mixing non mixing, has trivial centralizer
and singular spectrum with infinite multiplicity.
\end{prop}
\begin{proof}
The requirements of Theorem \ref{thm:Main} are satisfied since $T$
is ergodic, prime, preserves an infinite measure, and has property
CP as $\sigma^{*n}\perp\sigma^{*m}$ for all $n\neq m$.
This latter property also implies (see \cite{Roy07Infinite}) that each transformation
$S$ that commutes with $T_{*}$ is of the form $R_{*}$ for some transformation
$R$ of the base that commutes with $T$. Therefore, as the centralizer
of $T$ is trivial, so is the centralizer of $T_{*}$, and it follows that
$T_{*}$ is not rigid. As $T_{*}$ is prime, the only rigid factor
is the trivial one, consequently $T_{*}$ is mildly mixing.
To see that it is not mixing, it is sufficient to notice that
$\widehat{\sigma}\left(n_{j}\right)=\frac{1}{2}$ for all $j\ge0$.

By Proposition \ref{pro:singpowers}, $T_{*}$ has singular spectrum. It remains to prove that its multiplicity is infinite. Since $T$ has simple spectrum, $T_{*}$ is spectrally isomorphic to the dynamical system generated by the Gaussian process with spectral measure $\sigma$. By Girsanov's theorem, it has either infinite multiplicity or simple spectrum. In the latter case, for all $n\ge 1$ the map $\pi_{n}:\ \mathbb{T}^{n}\to\mathbb{T}$, $t_{1},\ldots,t_{n}\mapsto t_{1}+\cdots+t_{n}$ should be $n!$ to $1$ on some Borel set $F\subset \mathbb{T}^{n}$ with $\sigma^{\otimes n}(F)=1$ (for details, see e.g.\ \cite{KuPa12}).
However, this is impossible even for $n=2$ by a result of \cite{HoMePa86}.

Indeed it is shown there (chap IV, sect.\ 2.3) that for any positive Borel measure $\tau\ll\sigma$ the measure $\tau\ast\sigma$ is actually equivalent to $\sigma^{\ast 2}$. Given a $\sigma$-compact set $F\subset \mathbb{T}^{2}$ with $\sigma^{\otimes 2}(F)=1$, it follows that for every compact set $A\subset\mathbb{T}$ with $\sigma(A)>0$ we have $\sigma^{\ast 2}(\pi_{2}(F\cap (A\times\mathbb{T}))=1$. Hence, given any positive integer $k$, if we choose $k$ disjoint compact sets $A_{1},\ldots, A_{k}$ with $\sigma(A_{j})>0$ for all $j$, we obtain that $\cap_{j=1}^{k}\pi_{2}(F\cap (A_{j}\times\mathbb{T})\neq\emptyset$, and thus there are at least $k$ points in $F$ with the same image under $\pi_{2}$ (this proves actually that $U_{T_{*}}$ has infinite multiplicity on $H^{2}$).
\end{proof}

The fact that those systems possess a singular spectrum
of infinite multiplicity makes them new examples of prime systems.
Also:

\begin{cor}
The Poisson suspension $\left(X^{*},\mathcal{A}^{*},\mu^{*},T_{*}\right)$
is not a rank one system.
\end{cor}
\begin{proof}Indeed, it is mildly mixing and, in the Appendix, we give a short proof of the fact of independent interest that if a Poisson suspension is of rank one, then it is necessarily rigid.
\end{proof}

With the above examples we obtain a Poisson suspension with a continuum
array of non-disjoint, non-isomorphic prime factors:
\begin{prop}\label{prop:mutimeslambda}
The Poisson suspension
\[
\left(\left(X\times\left[0,1\right]\right)^{*},\left(\mathcal{A}\otimes\mathcal{B}\right)^{*},\left(\mu\otimes\lambda_{\left[0,1\right]}\right)^{*},\left(T\times\text{Id}\right)_{*}\right)
\]
possesses the Poisson suspensions $\left(X^{*},\mathcal{A}^{*},\left(c\mu\right)^{*},T_{*}\right)$,
$0<c\le1$ as factors. Those factors are prime, non-disjoint, unitarily
isomorphic and non metrically isomorphic for different values of $c$.\end{prop}
\begin{proof}
The factor relationship is implemented by the map $\nu\mapsto\nu\left(\cdot\times\left[0,c\right]\right)$.

The systems $\left(X,\mathcal{A},c\mu,T\right)$ have the same properties
as $c$ spans $[0,1]$. In particular $\left(X^{*},\mathcal{A}^{*},\left(c\mu\right)^{*},T_{*}\right)$
are prime and have the same spectrum, henceforth they all are unitarily
isomorphic.

It is recalled in Section \ref{sub:Superposition-of-Poisson-Measures}
that the superposition of two independent Poisson measures with intensity $\mu_{1}$
and $\mu_{2}$ leads to a Poisson measure with intensity $\mu_{1}+\mu_{2}$.
Therefore, if $c_{1}<c_{2}$ then $\left(X^{*},\mathcal{A}^{*},\left(c_{2}\mu\right)^{*},T_{*}\right)$
is a factor of the direct product of $\left(X^{*},\mathcal{A}^{*},\left(c_{1}\mu\right)^{*},T_{*}\right)$
with $\left(X^{*},\mathcal{A}^{*},\left(\left(c_{2}-c_{1}\right)\mu\right)^{*},T_{*}\right)$.
This yields a joining between $\left(X^{*},\mathcal{A}^{*},\left(c_{2}\mu\right)^{*},T_{*}\right)$
and $\left(X^{*},\mathcal{A}^{*},\left(c_{1}\mu\right)^{*},T_{*}\right)$;
it is not independent for obvious reasons.

Now assume there exists an isomorphism $S$ between both systems.
As $\sigma\perp\sigma^{*n}$, $n\ge2$, it implies, thanks to Proposition
5.2 in \cite{Roy07Infinite}, that $S=R^{*}$ for an isomorphism $R$
between $\left(X,\mathcal{A},c_{1}\mu,T\right)$ and $\left(X,\mathcal{A},c_{2}\mu,T\right)$,
but such an isomorphism does not exist by Proposition \ref{pro:AarNad}.
\end{proof}

\subsection{A mixing example}

Another source of examples is furnished by recent Ryzhikov's infinite
measure preserving ``mixing'' rank one transformations (see \cite{Ryz11mixrinfank}),
whose construction would be too long to be given here.

He proved, in particular, that all those systems have the \emph{minimal
self-joining} property in infinite measure (the only ergodic self-joinings
are off-diagonal joinings) which implies that they are prime as in
previous examples (see Proposition \ref{pro:AarNad}).
Moreover he proved, with some extra assumptions, that Poisson suspensions over
such systems have simple (and singular) spectrum,
which in turn implies that they have the property CP (indeed, a necessary
condition for a Poisson suspension to have simple spectrum is that
$\sigma^{*n}\perp\sigma^{*m}$ for all $n\neq m$, where $\sigma$
is the maximal spectral type of the base). If we sum up and apply
Theorem \ref{thm:Main}, we get:
\begin{prop}
There exist prime Poisson suspensions which are mixing, with simple
singular spectrum and trivial centralizer.
\end{prop}
Observe that, as any mixing rank one is MSJ(2), those prime mixing Poisson
suspensions are disjoint from any previously known prime systems,
thanks to Theorem \ref{thm:ParreauLemRoy} and Proposition \ref{prop:disjointness}.

\section{Appendix}
\begin{prop}
If a Poisson suspension is of rank one, then it is rigid.\end{prop}
\begin{proof}
We need the following property of rank one systems, established
by Ryzhikov in \cite{Ryz92Rank}: If $\Phi$ is a Markov operator
corresponding to an ergodic selfjoining of a rank one transformation
$T$, then there exists $a>0$, a Markov operator $\Psi$ and a sequence
$n_{k}$ such that $U_{T^{n_{k}}}$ converges weakly to $a\Phi+\left(1-a\right)\Psi$.

We recall that every Poisson suspension $\left(X^{*},\mathcal{A}^{*},\mu^{*},T_{*}\right)$
has the so-called ELF property (see \cite{Lem05ELF}), that
is every limit of off-diagonal joinings is ergodic. Therefore, in the
above situation, $a=1$ and $U_{T_{*}^{n_{k}}}\to \Phi$.

However, we cannot apply this result directly to $\Phi=Id$, as we have to rule out a sequence $(n_{k})$ that would be eventually equal to zero.

We recall moreover that we build a \emph{Poissonian joining} (see \cite{Lem05ELF}
and \cite{Roy07Infinite}) of a Poisson suspension $T_{*}$ by considering
a sub-Markov operator $\varphi$ that commutes with the base transformation $T$
and forming the exponential $\text{exp}(\varphi)$ that acts on each
chaos $H^{n}$ of $L^{2}\left(\mu^{*}\right)$ as $\varphi^{\odot n}$. Moreover, Poissonian joinings of an ergodic Poisson
suspension are ergodic. Therefore we can apply Ryzhikov's result to
the Markov operators $\text{exp}\left(\left(1-\frac{1}{n}\right)Id_{L^{2}\left(\mu\right)}\right)$
for each $n>1$. As $\left(1-\frac{1}{n}\right)Id_{L^{2}\left(\mu\right)}$
tends to $Id_{L^{2}\left(\mu^{*}\right)}$, we have that $\text{exp}\left(\left(1-\frac{1}{n}\right)Id_{L^{2}\left(\mu\right)}\right)$
tends to $\text{exp}\left(Id_{L^{2}\left(\mu\right)}\right)=Id_{L^{2}\left(\mu^{*}\right)}$.
It now follows that $Id_{L^{2}\left(\mu^{*}\right)}$ is a limit point of 
$(U_{T_{*}^{n}})$ ($n\neq 0$) and thus $T_{*}$ is rigid.
\end{proof}
\begin{rem}
Valery Ryzhikov informed us that the same proof shows even more, namely
that non-rigid Poisson suspensions are of local rank zero (and therefore
of infinite rank).\end{rem}

\bibliographystyle{plain}

\begin{thebibliography}{}

\end{thebibliography}


\begin{thebibliography}{10}

\bibitem{Aar87Eigen}
J.~Aaronson and M.~Nadkarni.
\newblock {$L_\infty$} eigenvalues and {$L_2$} spectra of non-singular
  transformations.
\newblock {\em Proc. London Math. Soc.}, 55:538--570, 1987.

\bibitem{BrMo74-Riesz}
G.~Brown and W.~Moran.
\newblock On orthogonality of {R}iesz products.
\newblock {\em Proc. Cambridge Phil. Soc.}, 76:173--181, 1974.

\bibitem{deljunco83counter}
A.~del Junco.
\newblock A family of counterexamples in ergodic theory.
\newblock {\em Israel J}, 44(2):160--188, 1983.

\bibitem{DelDan08Cut}
A.~del Junco and Danilenko~A. I.
\newblock Cut-and-stack simple weakly mixing map with countably many prime
  factors.
\newblock {\em Proc}, 136:2463--2472, 2008.

\bibitem{JunRaSwa80Chacon}
A.~del Junco, M.~Rahe, and L.~Swanson.
\newblock Chacon's automorphism has minimal self-joinings.
\newblock {\em J. Analyse Math.}, 37:276--284, 1980.

\bibitem{JuncRud87graphs}
A.~del Junco and D.~J. Rudolph.
\newblock On ergodic actions whose self-joinings are graphs.
\newblock {\em Ergodic Theory Dynam. Systems}, 7(4):531--557, 1987.

\bibitem{DelRu87Rankonerig}
A.~del Junco and D.~J. Rudolph.
\newblock A rank one, rigid, simple, prime map.
\newblock {\em Ergodic Theory Dynam. Systems}, 7:229--247, 1987.

\bibitem{Lem05ELF}
Y.~Derriennic, K.~Fr{\c{a}}czek, M.~Lema$\acute{\mathrm{n}}$czyk, and
  F.~Parreau.
\newblock Ergodic automorphisms whose weak closure of off-diagonal measures
  consists of ergodic self-joinings.
\newblock {\em Colloq. Math.}, 110:81--115, 2008.

\bibitem{GlaWei94Prime}
E.~Glasner and B.~Weiss.
\newblock A simple weakly mixing transformation with nonunique prime factors.
\newblock {\em Amer. J. Math}, 116(2):361--375, 1994.

\bibitem{HoMePa86}
B.~Host, M\'{e}la, and F.~Parreau.
\newblock Analyse harmonique des mesures.
\newblock In {\em Ast{\'e}risque}, volume 135-136. S.M.F., 1986.

\bibitem{HoMePa91nonsing}
B.~Host, J.-F. M\'{e}la, and F.~Parreau.
\newblock Non-singular transformations and spectral theory.
\newblock {\em Bull. Soc. Math. France}, 119:33--90, 1991.

\bibitem{King86Comweak}
J.~L. King.
\newblock The commutant is the weak closure of the powers, for rank-1
  transformations.
\newblock {\em Ergodic Theory Dynam. Systems}, 6(3):363--384, 1986.

\bibitem{King88Struc}
J.~L. King.
\newblock Joining-rank and the structure of finite-rank mixing transformations.
\newblock {\em J. A}, 51:182--227, 1998.

\bibitem{KuPa12}
J.~Ku{\l}aga-Przymus and F.~Parreau.
\newblock Disjointness properties for cartesian products of weakly mixing
  systems.
\newblock {\em Colloq. Math.}, 128(2):153--177, 2012.

\bibitem{LastPen11Fock}
G.~Last and M.~Penrose.
\newblock Poisson process {F}ock space representation, chaos expansion and
  covariance inequalities.
\newblock {\em Probab. Theory Related Fields}, 150(3-4):663--690, 2011.

\bibitem{LemPaRo11JoinPrim}
M.~Lema$\acute{\mathrm{n}}$czyk, F.~Parreau, and E.~Roy.
\newblock Joining {P}rimeness and disjointness from infinitely divisible
  systems.
\newblock {\em Proc. Amer. Math. Soc.}, 139:185--199, 2011.

\bibitem{LemParThou00Gausselfjoin}
M.~Lema$\acute{\mathrm{n}}$czyk, F.~Parreau, and J.-P. Thouvenot.
\newblock Gaussian automorphisms whose ergodic self-joinings are {G}aussian.
\newblock {\em Fund. Math.}, 164:253--293, 2000.

\bibitem{Lieb94Isom}
V.~Liebscher.
\newblock On the {I}somorphism of {P}oisson {S}pace and {S}ymmetric {F}ock
  {S}pace.
\newblock In {\em Quantum Probability and Related Topics}. Accardi, L., 1994.

\bibitem{Meck67Form}
J.~Mecke.
\newblock Station\"{a}re zuf\"{a}llige ma{\ss}e auf lokalkompakten abelschen
  gruppen.
\newblock {\em Z. Wahrsch. Verw. Gebiete.}, 9:36--58, 1967.

\bibitem{Orns72root}
D.~S. Ornstein.
\newblock On the root problem in ergodic theory.
\newblock In Univ. of~Calif.~Press, editor, {\em Proc. Sixth Berkeley Symp. on
  Math. Statist. and Prob.}, volume~2, pages 347--356, 1972.

\bibitem{Peyr74Riesz}
J.~Peyri\`{e}re.
\newblock Etude de quelques propri\'{e}t\'{e}s des produits de {R}iesz.
\newblock {\em Ann. Inst. Fourier}, 25(2):127--169, 1975.

\bibitem{Polit74weakiso}
S.~Polit.
\newblock {\em Weakly Isomorphic Maps Need Not Be Isomorphic}.
\newblock PhD thesis, Stanford, 1974.

\bibitem{Roy07Infinite}
E.~Roy.
\newblock Poisson suspensions and infinite ergodic theory.
\newblock {\em Ergodic Theory Dynam. Systems}, 29(2):667--683, 2009.

\bibitem{Rud90Fund}
D.~J. Rudolph.
\newblock {\em Fundamentals of Measurable Dynamics}.
\newblock Oxford University press, 1990.

\bibitem{Ryz92Rank}
V.~V. Ryzhikov.
\newblock Mixing, rank and minimal self-joining of actions with invariant
  measure.
\newblock {\em Mat. Sb.}, 183(3):133--160, 1992.

\bibitem{Ryz11mixrinfank}
V.~V. Ryzhikov.
\newblock On mixing rank one infinite transformations.
\newblock preprint, 2011.

\bibitem{Thou95Joinings}
J.-P. Thouvenot.
\newblock Some properties and applications of joinings in ergodic theory.
\newblock In {\em Ergodic theory and its connections with harmonic analysis
  (Alexandria, 1993)}, volume 205 of {\em London Math. Soc. Lecture Note Ser.},
  pages 207--235. Cambridge Univ. Press, Cambridge, 1995.

\bibitem{Veech82prime}
W.~A. Veech.
\newblock A criterion for a process to be prime.
\newblock {\em Monatsh. Math.}, 94(4):335--341, 1982.

\end{thebibliography}

\end{document}